\documentclass[11pt,reqno]{amsart}
\usepackage{amsmath, amssymb, amsthm}
\usepackage{url}
\usepackage[breaklinks]{hyperref}
\usepackage{lipsum}
\usepackage{curve2e}
\usepackage{microtype}
\usepackage[utf8]{inputenx}
\definecolor{gray}{gray}{0.4}

\usepackage{mathptmx}

\usepackage{amssymb,amsthm}
\usepackage{dsfont}
\usepackage{mathrsfs}
\usepackage{a4wide}
\usepackage{enumerate}
\usepackage{tfrupee}
\allowdisplaybreaks
\renewcommand{\k}{\kappa}

\renewcommand{\(}{\left\(}
\renewcommand{\)}{\right\)}
\renewcommand{\[}{\left\[}
\renewcommand{\]}{\right\]}
\renewcommand{\i}{\infty}
\numberwithin{equation}{section}
\theoremstyle{plain}
\newtheorem{theorem}{Theorem}[section]
\newtheorem{lemma}[theorem]{Lemma}
\newtheorem{remark}[]{Remark}

\newtheorem{corollary}[theorem]{Corollary}

\makeatletter
\def\proof{\@ifnextchar[{\@oproof}{\@nproof}}
\def\@oproof[#1][#2]{\trivlist\item[\hskip\labelsep\textit{#2 Proof of\
		#1.}~]\ignorespaces}
\def\@nproof{\trivlist\item[\hskip\labelsep\textit{Proof.}~]\ignorespaces}

\makeatother

\begin{document}
	\title[A finite analogue of a $q$-series identity]{A finite analogue of a $q$-series identity of Bhoria, Eyyunni and Maji and its applications} 
	
	\author{Atul Dixit and Khushbu Patel}\thanks{2020 \textit{Mathematics Subject Classification.} Primary 11P81, 33D15; Secondary 11P84, 05A17.\\
		\textit{Keywords and phrases.} partitions, finite analogues, smallest parts function, rank, crank.}
	\address{Discipline of Mathematics, Indian Institute of Technology Gandhinagar, Palaj, Gandhinagar 382355, Gujarat, India} 
	\email{adixit@iitgn.ac.in; khushbup@iitgn.ac.in}
	\begin{abstract}
		Bhoria, Eyyunni and Maji recently obtained a four-parameter $q$-series identity which gives as special cases not only all five entries of Ramanujan on pages 354 and 355 of his second notebook but also allows them to obtain an analytical proof of a result of Bressoud and Subbarao. Here, we obtain a finite analogue of their identity which naturally gives finite analogues of Ramanujan's results. Using one of these finite analogues, we deduce an identity for a finite sum involving a ${}_2\phi_1$. This identity is then applied to obtain a generalization of the generating function version of Andrews' famous identity for the smallest parts function $\textup{spt}(n)$. The $q$-series which generalizes $\sum_{n=1}^{\infty}\textup{spt}(n)q^n$ is completely different from $S(z, q)$ considered by Andrews, Garvan and Liang. Further applications of our identity are given. Lastly we generalize a result of Andrews, Chan and Kim which involves the first odd moments of rank and crank. 
	\end{abstract}
\maketitle
\section{Introduction}\label{intro}
In \cite[p.~354]{ramanujanoriginalnotebook2}, \cite[p.~263, Entry 3]{bcbramforthnote}, Ramanujan gave the following beautiful identity for $a\neq 0, |a|<1,$ and $|b|<1$, namely,
\begin{align}\label{neglected}
	\sum_{n=1}^{\infty} \frac{ (b/a)_n a^n }{ (1- q^n) (b)_n } = \sum_{n=1}^{\infty} \frac{a^n - b^n }{ 1- q^n },
\end{align}
where, here and throughout the paper, we have used the standard $q$-series notation:
\begin{align*}
	(a)_0 &:=(a;q)_0 =1, \qquad \\
	(a)_n &:=(a;q)_n  = (1-a)(1-aq)\cdots(1-aq^{n-1}),
	\qquad n \geq 1, \\
	(a)_{\infty} &:=(a;q)_{\i}  = \lim_{n\to\i}(a;q)_n, \qquad |q|<1.
\end{align*}
Also, we will always consider $q\in\mathbb{C}$ such that $|q|<1$.

Maji and the first author \cite{dixitmaji18} obtained the following generalization of \eqref{neglected}:
\begin{theorem}\label{gen of Ramanujan's identity}
	Let $a, b, c$ be three complex numbers such that $|a|<1$ and $|cq|< 1$. Then
	\begin{align}\label{entry3gen}
		\sum_{n =1}^{\infty}  \frac{ (b/a)_n a^n }{ (1- c q^n) (b)_n } =  \sum_{m=0}^{\infty}\frac{(b/c)_mc^m}{(b)_m}\left(\frac{aq^m}{1-aq^m}-\frac{bq^m}{1-bq^m}\right).
	\end{align}
	Moreover, for $|a|<1$ and $|b|<\min(|c|,1)$,
	\begin{align*}
		\sum_{n =1}^{\infty}  \frac{ (b/a)_n a^n }{ (1- c q^n) (b)_n }=\frac{ ( b/c )_{\infty }}{(b)_{\infty} }   \sum_{n=0}^{\infty} \frac{(c)_n ( b/c)^n }{(q)_n}    \sum_{m=1}^{\infty} \frac{a^m - b^m }{1- c q^{m+n}  }.
	\end{align*}
\end{theorem}
Theorem \ref{gen of Ramanujan's identity} has many nice implications in partition theory which are discussed in \cite{dixitmaji18}. These include, in particular,  a new proof of the generating function version of Andrews' famous identity
\begin{equation*}
	\textup{spt}(n)=np(n)-\frac{1}{2}N_2(n).
\end{equation*}
Clearly, Ramanujan's identity \eqref{neglected} is the special case $c=1$ of \eqref{entry3gen}.

In a recent work \cite{dems}, Eyyunni, Maji, Sood and the first author obtained a finite analogue of Theorem \ref{gen of Ramanujan's identity} given below.
\begin{theorem}\label{finmainabc}
	Let $N\in\mathbb{N}$. For $a, b, c\neq q^{-n}, 1\leq n\leq N-1$, $c\neq q^{-N}$, and $a, b\neq 1$,
	\begin{align*}
		\sum_{n=1}^{N}\left[\begin{matrix} N\\n\end{matrix}\right]\frac{(\frac{b}{a})_{n}(q)_{n}(a)_{N-n}a^{n}}{(1-cq^{n})(b)_n(a)_N}=
		\sum_{n=1}^{N}\left[\begin{matrix} N\\n\end{matrix}\right]\frac{(\frac{b}{c})_{n-1}(q)_n (cq)_{N-n}c^{n-1}}{(b)_{n-1}(cq)_N}\left(\frac{aq^{n-1}}{1-aq^{n-1}}-\frac{bq^{n-1}}{1-bq^{n-1}}\right),
	\end{align*}
	where
	\begin{align*}
		\left[\begin{matrix} N\\n\end{matrix}\right]=\left[\begin{matrix} N\\n\end{matrix}\right]_q :=\begin{cases}
			\frac{(q;q)_N}{(q;q)_n (q;q)_{N-n}},\hspace{2mm}\text{if}\hspace{1mm}0\leq n\leq N,\\
			0,\hspace{2mm}\text{otherwise},
		\end{cases} 
	\end{align*}
	is the $q$-binomial coefficient.
\end{theorem}
Letting $N\to\infty$ in the above theorem leads to \eqref{entry3gen}. The finite analogue is a generalization of \eqref{entry3gen}, for, it is valid for \emph{any} natural number $N$ apart from being valid in the limiting case $N\to\infty$.  A finite analogue is desirable whenever possible since it can be extended to the level of elliptic hypergeometric series. Elliptic extensions are essentially always finite, balanced and well-poised. Only finite analogues have chance to be extended to the elliptic setting.

It was shown in \cite{dems} that Theorem \ref{finmainabc} has many fruitful consequences in the theory of partitions including a representation for $\textup{spt}(n, N)$, the number of smallest parts in all partitions of $n$ whose corresponding largest parts are less than or equal to $N$, in terms of the restricted partition function $p(n, N)$, which counts the number of partitions of $n$ whose largest parts do not exceed $N$, and a finite analogue of the second Atkin-Garvan rank moment. See Theorem 2.4 of \cite{dems}. The work in \cite{dems} further extended the sparsely developed theory of $p(n, N)$. Note that $\lim_{N\to\infty}p(n, N)=p(n)$, the number of unrestricted partitions of a positive integer $n$.

Very recently, Bhoria, Eyyunni and Maji \cite[Theorem 2.1]{bem} further generalized Theorem \ref{gen of Ramanujan's identity} by obtaining the following result.
\begin{theorem}\label{Generalization of Dixit-Maji}
	Let $a, b, c, d$  be four complex numbers such that $|ad|<1$ and $|cq|<1$. Then
	\begin{equation*}
		\sum_{n = 1}^{\infty}  \frac{ (b/a)_n (c/d)_n (ad)^n }{ (b)_n (cq)_n } = 
		\frac{(a-b)(d-c)}{(ad-b)}\sum_{m=0}^{\infty}\frac{(a)_m (bd/c)_m c^m}{(b)_m (ad)_m}
		\left(\frac{adq^m}{1-adq^m}-\frac{bq^m}{1-bq^m}\right).
	\end{equation*}
\end{theorem}
 This theorem also has many beautiful consequences in $q$-series and partition theory. First of all, from this single identity, the authors of \cite{bem} could derive all five entries in the unorganized portion of Ramanujan's second and third notebooks \cite[pp.~354--355]{ramanujanoriginalnotebook2} (see also \cite[pp.~302--303]{ramanujanoriginalnotebook2tifr}). These entries can be found in Section \ref{fafive} of our paper. Moreover, using a differential operator to act on the partition-theoretic interpretation of the fourth of these entries, they were able to derive an important result of Bressoud and Subbarao \cite{bresub} which had defied a proof using analytical techniques for about 38 years. (Bressoud and Subbarao's proof was combinatorial). Moreover, the authors of \cite{bem} even generalized the result of Bressoud and Subbarao.
 
 In this paper, we obtain a finite analogue of Theorem \ref{Generalization of Dixit-Maji}, that is, the identity of Bhoria, Eyyunni and Maji. This allows to obtain finite analogues of all five entries of Ramanujan. We also derive a result involving a finite sum of a ${}_2\phi_1$ which has several important corollaries, for example, a generalization of the generating function version of Andrews' identity  \eqref{idspt} as well as of an identity involving the function $N_{\textup{SC}}(n)$ (see Section 5.2 for a definition). We also generalize an identity of Andrews, Chan and Kim from \cite{agl13} by considering finite analogues of first odd moments of rank and crank.
 
 \section{A finite analogue of a four parameter $q$-series identity}
 \begin{theorem}\label{fafour}
 	We have
 	\begin{align}\label{eq:9}
 		&\sum_{n=1}^{N}\left[\begin{array}{c}N\\n\end{array}\right] \frac{(q)_n\left(\frac{b}{a}\right)_n\left(\frac{c}{d}\right)_n(ad)_{N-n}(ad)^n}{(b)_n(cq)_n(ad)_{N}}\nonumber\\
 	&=\frac{(a-b)(d-c)}{(ad-b)}\displaystyle\sum_{n=1}^{N} \left[\begin{array}{c}N\\n\end{array}\right]\frac{(a)_{n-1}\left(\frac{bd}{c}\right)_{n-1}(q)_n\left(cq\right)_{N-n} c^{n-1}}{(b)_{n-1}(cq)_N(ad)_{n-1}}
 			\left(\frac{adq^{n-1}}{1-adq^{n-1}} - \frac{bq^{n-1}}{1-bq^{n-1}}\right).
 		\end{align}
 \end{theorem}
\begin{proof}
	From \cite[p. 70, Equation (3.2.1)]{gasperrahman},
	\begin{align*}
		{}_4\phi_{3}\left[\begin{array}{cccccc} q^{-N},&A,&B,&C\;\;\; ;&q\;,&q\\D,&E,&\frac{ABCq^{1-N}}{DE}& & & \end{array}\right]
		=\frac{\left(\frac{E}{A}\right)_{N}\left(\frac{DE}{BC}\right)_{N}}{\left(E\right)_{N}\left(\frac{DE}{ABC}\right)_{N}}{}_4\phi_{3}\left[\begin{array}{cccccc} q^{-N},&A,&\frac{D}{B},&\frac{D}{C}\;\;\; ;&q\;,&q\\D,&\frac{DE}{BC},&\frac{Aq^{1-N}}{E}& & & \end{array}\right].
	\end{align*}
Let $A=q,\;B=\displaystyle\frac{bq}{a},\;C=\frac{cq}{d},\;D=bq$ and $E=cq^2$ so as to have
	\begin{equation}\label{eq:10}
		\sum_{n=0}^{N}\frac{\left(q^{-N}\right)_{n}\left(\frac{bq}{a}\right)_{n}\left(\frac{cq}{d}\right)_{n}}{\left(bq\right)_{n}\left(\frac{q^{1-N}}{ad}\right)_{n}\left(cq^2\right)_{n}}q^{n}
		=\displaystyle \frac{(1-cq)(1-adq^N)}{(1-cq^{N+1})(1-ad)} \displaystyle\sum_{n=0}^{N} \frac{\left(q^{-N}\right)_n (a)_n \left(\frac{bd}{c}\right)_n}{(bq)_n (adq)_n \left(\frac{q^{-N}}{c}\right)_{n}}q^n.
	\end{equation}
	Using the elementary formula \cite[p. 15, Eq (4.1)]{dems}
	\begin{align}\label{ef}
	\left(\frac{q^{-N}}{x}\right)_{n} = \frac{(-1)^n (xq^{N-n+1})_n}{x^n q^{Nn}}q^{\frac{n(n-1)}{2}}
	\end{align}
to derive
\begin{align}
		\frac{\left(q^{-N}\right)_{n}}{\left(\frac{q^{-N}}{x}\right)_{n}}=\frac{\left(q^{N-n+1}\right)_{n}}{\left(xq^{N-n+1}\right)_{n}} x^{n}=\frac{\left(q\right)_{N}\left(xq\right)_{N-n}}{\left(q\right)_{N-n}\left(xq\right)_{N}}x^{n},\label{eq:11}
	\end{align}
 and then using the latter, with $x=ad$ for the left-hand side of \eqref{eq:10}, and then again with $x=c$ for the right-hand side, we see that
	\begin{equation*}\sum_{n=0}^{N}\frac{\left(\frac{bq}{a}\right)_{n}\left(\frac{cq}{d}\right)_{n}\left(ad\right)_{N-n}\left(q\right)_{N}}{\left(bq\right)_{n}\left(cq\right)_{n+1}\left(ad\right)_{N+1}\left(q\right)_{N-n}}\left(ad\right)^{n}=\sum_{n=0}^{N}\frac{\left(a\right)_{n}\left(\frac{bd}{c}\right)_{n}\left(q\right)_{N}\left(cq\right)_{N-n}(cq)^{n}}{\left(bq\right)_{n}\left(ad\right)_{n+1}\left(q\right)_{N-n}\left(cq\right)_{N+1}}\end{equation*}
	Replace $n$ by $n-1$ on both sides and multiply both sides of the resulting equation by  $ \frac{ad(1-\frac{b}{a})(1-\frac{c}{d})(1-q^{N+1})}{(1-b)}$ to get
	\begin{flushleft}
		$\displaystyle\sum_{n=1}^{N+1} \frac{\left(\frac{b}{a}\right)_n \left(\frac{c}{d}\right)_n (ad)_{N+1-n}(q)_{N+1}}{(b)_n (cq)_n (ad)_{N+1}(q)_{N+1-n}}(ad)^n \displaystyle ={(a-b)(d-c)}\displaystyle\sum_{n=1}^{N+1} \frac{(a)_{n-1} \left(\frac{bd}{c}\right)_{n-1}(q)_{N+1}(cq)_{N+1-n}}{{(b)_n(ad)_n (q)_{N+1-n} (cq)_{N+1}}}(cq)^{n-1}.$
	\end{flushleft}
	Replace $N$ by $N-1$ so that
\begin{align*}
&\sum_{n=1}^{N} \frac{\left(\frac{b}{a}\right)_n \left(\frac{c}{d}\right)_n (ad)_{N-n} (q)_N}{(b)_n (cq)_n (ad)_N (q)_{N-n}}(ad)^n \nonumber\\
& =\displaystyle\frac{(a-b)(d-c)}{(ad-b)} \displaystyle\sum_{n=1}^{N} \left(\frac{(a)_{n-1} \left(\frac{bd}{c}\right)_{n-1} (q)_N (cq)_{N-n} c^{n-1}}{(b)_{n-1} (ad)_{n-1} (q)_{N-n} (cq)_N}\frac{(ad-b) q^{n-1} }{(1-bq^{n-1}) (1-adq^{n-1})}\right).
\end{align*}
	This results in \eqref{eq:9} upon writing 
	\begin{equation*}
		\frac{(ad-b) q^{n-1} }{(1-bq^{n-1}) (1-adq^{n-1})}= \frac{adq^{n-1}}{1-adq^{n-1}}- \frac{bq^{n-1}}{1-bq^{n-1}}.
	\end{equation*}
\end{proof} 

\begin{corollary}
	\begin{equation}\label{eq:13}
		\sum_{n=1}^{N}\left[\begin{array}{c}N\\n\end{array}\right] \frac{(q)_n \left(\frac{c}{d}\right)_n (-zd)^n q^\frac{n(n+1)}{2}}{(zq)_n (cq)_n}
		= \frac{z}{c}(c-d)\displaystyle\sum_{n=1}^{N}\left[\begin{array}{c}N\\n\end{array}\right] \frac{(q)_n \left(\frac{zdq}{c}\right)_{n-1} (cq)_{N-n} (cq)^n}{(zq)_n (cq)_N}.\end{equation}
\end{corollary}
\begin{proof}
	Let $a \longrightarrow 0 , b \longrightarrow zq$ in \eqref{eq:9} and simplify using the fact that
	\begin{align*}
		\lim\limits_{a\to0}\left(\frac{zq}{a}\right)_n (a)^n=\lim\limits_{a\to0}\left(1-\frac{zq}{a}\right)\cdots \left(1-\frac{zq^{n}}{a}\right) (a)^n
		=(-1)^n z^n q^\frac{n(n+1)}{2}.
	\end{align*}
\end{proof}
\begin{corollary}
	\begin{equation}\label{eq:14}
		\displaystyle\sum_{n=1}^{N}\left[\begin{array}{c}N\\n\end{array}\right] \frac{(q)_n (zc)^n q^{n^2}}{(zq)_n (cq)_n} = z \displaystyle\sum_{n=1}^{N} \left[\begin{array}{c}N\\n\end{array}\right]\frac{(q)_n (cq)_{N-n} (cq)^n}{(zq)_n (cq)_N}.
	\end{equation}
\end{corollary}
\begin{proof}
	Let $d\longrightarrow 0$ in \eqref{eq:13}.  
\end{proof}
\begin{remark}
Letting $c=1/z$ in the above identity, we get 
	\begin{equation*}
	\sum_{n=1}^{N}\left[\begin{array}{c}N\\n\end{array}\right] \frac{(q)_n q^{n^2}}{(zq)_n(z^{-1}q)_n} = z\sum_{n=1}^{N}\left[\begin{array}{c}N\\n\end{array}\right] \frac{(q)_n(z^{-1}q)_{N-n}(z^{-1}q)^n}{(zq)_n(z^{-1}q)_N} ,
	\end{equation*}
where $\displaystyle\sum_{n=1}^{N} \left[\begin{array}{c}N\\n\end{array}\right]\frac{(q)_n q^{n^2}}{(zq)_n(z^{-1}q)_n}$ 
is a finite analogue of the rank generating functions \cite[Theorem 2.2]{dems}. If we let $N\to\infty$ in \eqref{eq:14}, we recover an identity of Andrews \cite[p.~24, Corollary 2.2]{yesto}:
\begin{equation}\label{d=0analogue_THM2.2_DM}
	\sum_{n=1}^{\infty} \frac{ z^n  c^n q^{n^2}}{ (z q)_n (c q)_n } = z \sum_{n=1}^{\infty} \frac{(c q)^n }{(z q)_n}. 
\end{equation}
\end{remark}

\section{Finite analogues of five entries of Ramanujan}\label{fafive}
Our aim of this section is to derive the finite analogues of the five entries of Ramanujan mentioned in the introduction. The first of the five is
\begin{equation}\label{entry1}
	\frac{(-aq)_{\infty}}{(bq)_{\infty}}=\sum_{n=0}^{\infty}\frac{(-b/a)_na^nq^{n(n+1)/2}}{(q)_n(bq)_n}.
\end{equation}
Ramanujan has also recorded in his Lost Notebook \cite[p.~370]{lnb}. The finite analogue of this entry is
\begin{theorem}
(Finite analogue of  Entry 1)\\
The following identity holds:
	\begin{equation*}
	\sum_{n=0}^{N}\left[\begin{array}{c}N\\n\end{array}\right] \frac{\left(\frac{-b}{a}\right)_na^n q^\frac{n(n+1)}{2}}{(bq)_n}=\displaystyle\sum_{n=0}^{N}\left[\begin{array}{c}N\\n\end{array}\right]\frac{\left(\frac{-a}{b}\right)_n(bq)_{N-n} (bq)^n}{(bq)_N}.
	\end{equation*}
\end{theorem}
\begin{proof}
Let  $z=1, d=-a$ and $c=b$ in \ref{eq:13}.
\end{proof}
Note that if we let $N\to\infty$ in the above theorem and use the $q$-binomial theorem on the resulting right-hand side, we obtain \eqref{entry1}.
Ramanujan's Entry 2 is
 \begin{equation*}
	(aq)_{\infty}\sum_{n=1}^{\infty}\frac{na^nq^{n^2}}{(q)_n(aq)_n}=\displaystyle\sum_{n=1}^{\infty}\frac{(-1)^{n-1}a^nq^\frac{n(n+1)}{2}}{1-q^n}.
\end{equation*}
The proof of this identity given in \cite{bem} begins with \eqref{d=0analogue_THM2.2_DM} and is quite long (see p.~446--449 of \cite{bem}). In what follows, we give a short proof of a finite analogue of this entry. Then letting $N\to\infty$ in this finite analogue immediately gives the entry itself.
\begin{theorem}
	(Finite analogue of  Entry 2)\\
We have
	\begin{equation}\label{correct}
		(aq)_N\sum_{n=1}^{N}\left[\begin{array}{c}N\\n\end{array}\right]\frac{na^nq^{n^2}}{(aq)_n}=\displaystyle\sum_{n=1}^{N}\left[\begin{array}{c}N\\n\end{array}\right]\frac{(q)_n(-1)^{n-1}a^nq^\frac{n(n+1)}{2}}{1-q^n}.
		\end{equation}
\end{theorem}    
\begin{proof}
	Andrews' finite Heine transformation \cite[Cor. 3]{andfinheine} is
	\begin{equation*}
		{}_{3}\phi_{2}\left[\begin{array}{ccccc} q^{-N},&a,&b\;\;\; ;&q\;,&q\\c,&\frac{q^{1-N}}{t}& & & \end{array}\right]=\frac{\left(\frac{c}{b}\right)_{N}\left(bt\right)_{N}}{\left(c\right)_{N}\left(t\right)_{N}}\;{}_{3}\phi_{2}\left[\begin{array}{ccccc} q^{-N},&\frac{abt}{c},b,\;\;\; ;&q\;,&q\\bt,&\frac{bq^{1-N}}{c} & & & \end{array}\right],
		\end{equation*}
	that is,
	\begin{equation*}
		\sum_{n=0}^{N}\frac{(q^{-N})_n(a)_n(b)_n}{(c)_n\left(\frac{q^{1-N}}{t}\right)_n(q)_n}q^n =\frac{\left(\frac{c}{b}\right)_N(bt)_N}{(c)_N(t)_N} \displaystyle\sum_{n=0}^{N}\frac{(q^{-N})_n\left(\frac{abt}{c}\right)_n(b)_n}{\left(\frac{bq^{1-N}}{c}\right)_n(bt)_n(q)_n}q^n.
		\end{equation*}
Employing \eqref{eq:11}, and then replacing $t$ by $\frac{t}{ab}$ , we get 
	\begin{equation*}
		\sum_{n=0}^{N}\frac{(a)_n(b)_n(q)_N\left(\frac{t}{ab}\right)_{N-n}}{(c)_n(q)_n(q)_{N-n}\left(\frac{t}{ab}\right)_{N}}\left(\frac{t}{ab}\right)^n=\frac{\left(\frac{c}{b}\right)_N\left(\frac{t}{a}\right)_N}{(c)_N\left(\frac{t}{ab}\right)_N}\displaystyle\sum_{n=0}^{N}\frac{\left(\frac{t}{c}\right)_n(b)_n(q)_N\left(\frac{c}{b}\right)_{N-n}}{\left(\frac{t}{a}\right)_{n}(q)_n(q)_{N-n}\left(\frac{c}{b}\right)_N}\left(\frac{c}{b}\right)^n.
		\end{equation*}
	Now let $a,b \to \infty$ and  use the fact $\lim_{x\to \infty}(x)_{n}/x^n=(-1)^nq^\frac{n(n-1)}{2}$ twice to derive
	\begin{equation*}
		\sum_{n=0}^{N}\left[\begin{array}{c}N\\n\end{array}\right] \frac{t^nq^{n(n-1)}}{(c)_n}=\frac{1}{(c)_N} \sum_{n=0}^{N}\left[\begin{array}{c}N\\n\end{array}\right]\left(\frac{t}{c}\right)_n(-1)^nc^nq^\frac{n(n-1)}{2}.
		\end{equation*}
	Next, replace $t$ by $bq$ and $c$ by $aq$ so that
	\begin{equation*}
		\sum_{n=0}^{N}\left[\begin{array}{c}N\\n\end{array}\right] \frac{b^nq^{n^2}}{(aq)_n}=\frac{1}{(aq)_N} \sum_{n=0}^{N}\left[\begin{array}{c}N\\n\end{array}\right]\left(\frac{b}{a}\right)_n(-1)^na^nq^\frac{n(n+1)}{2}.
		\end{equation*} 
	Now differentiate this identity with respect to $b$ to get
	\begin{equation*}
		(aq)_N\sum_{n=1}^{N}\left[\begin{array}{c}N\\n\end{array}\right] \frac{nab^{n-1}q^{n^2}}{(aq)_n}=\displaystyle \sum_{n=1}^{N}\left[\begin{array}{c}N\\n\end{array}\right]\left(\frac{b}{a}\right)_n(-1)^{n-1}a^nq^\frac{n(n+1)}{2}\sum_{k=0}^{n-1}\left(\frac{q^k}{1-bq^k/a}\right).
		\end{equation*}
	Now let $ b \to a $ and note that
	\begin{align*}
		\lim\limits_{b\to a}\left(\frac{b}{a}\right)_{n}\sum_{k=0}^{n-1}\left(\frac{q^k}{1-\frac{bq^k}{a}}\right) 
		&=\lim\limits_{b\to a}\left[\left(\frac{b}{a}\right)_n\frac{1}{1-\frac{b}{a}} + \left(\frac{b}{a}\right)_n\left(\frac{q}{1-\frac{bq}{a}}+\cdots+\frac{q^{n-1}}{1-\frac{bq^{n-1}}{a}}\right)\right]\\
		&=(q)_{n-1}. 
	\end{align*}
This results in \eqref{correct}.
\end{proof}    
The third entry of Ramanujan is
\begin{equation}
\label{entry3}
	\sum_{n=1}^{\infty}\frac{(\frac{b}{a})_na^n}{(1-q^n)(b)_n}=\sum_{n=1}^{\infty}\frac{a^n-b^n}{1-q^n}.
	\end{equation}
We now give its finite analogue.
\begin{theorem}
(Finite analogue of  Entry 3)\\
	We have
\begin{equation}\label{entry3fin}
\sum_{n=1}^{N}\left[\begin{array}{c}N\\n\end{array}\right] \frac{(q)_n\left(\frac{b}{a}\right)_n(a)_{N-n}a^n}{(b)_n(1-q^n)(a)_{N}}=\displaystyle\sum_{m=1}^{\infty}\frac{(a^m-b^m)(1-q^{mN})}{1-q^m}.
		\end{equation}
\end{theorem}
\begin{proof}
	Divide both sides of \eqref{eq:9} by $1-c/d$ and then let $c=d=1$ thereby obtaining
	\begin{align}\label{entry3int}
		\sum_{n=1}^{N}\left[\begin{array}{c}N\\n\end{array}\right] \frac{(q)_n\left(\frac{b}{a}\right)_n(a)_{N-n}a^n}{(b)_n(1-q^n)(a)_{N}}
	&=\sum_{n=1}^{N}\left(\frac{aq^{n-1}}{1-aq^{n-1}}-\frac{bq^{n-1}}{1-bq^{n-1}}\right)\nonumber\\
	&=\sum_{m=1}^{\infty}(a^m-b^m)\sum_{n=0}^{N-1}(q^n)^m\nonumber\\
	&=\sum_{m=1}^{\infty}\frac{(a^m-b^m)(1-q^{mN})}{1-q^m}.
	\end{align}
\end{proof}    
The limiting case $N \longrightarrow \infty $ of \eqref{entry3fin} gives  \eqref{entry3}.

\begin{theorem}
	(Finite analogue of  Entry 4)\\
The following identity is valid:
\begin{equation*}
	\sum_{n=1}^{N}\left[\begin{array}{c}N\\n\end{array}\right]\frac{(-1)^{n-1}a^nq^\frac{n(n+1)}{2}(q)_n}{(1-q^n)(aq)_n}=\sum_{n=1}^{N}\frac{aq^n}{1-aq^n}.
		\end{equation*}
\end{theorem}    
\begin{proof}
	Let $a \longrightarrow 0$ in the first equality of \eqref{entry3int} and then replace $b$ by $aq$.
\end{proof}
The limiting case $N \longrightarrow \infty $, in particular, gives Entry 4:
\begin{equation}\label{entry4}
	\sum_{n=1}^{\infty}\frac{(-1)^{n-1}a^nq^\frac{n(n+1)}{2}}{(1-q^n)(aq)_n}=\sum_{n=1}^{\infty}\frac{a^nq^n}{1-q^n}.
	\end{equation}
\begin{theorem}
	(Finite analogue of  Entry 5)\\
	We have
\begin{equation}\label{entry5fin}
\sum_{n=1}^{N}\left[\begin{array}{c}N\\n\end{array}\right] \frac{(q)_n(q)_{n-1}(a)_{N-n}a^n}{(a)_n(1-q^n)(a)_{N}}=\displaystyle\sum_{n=1}^{N}\frac{aq^{n-1}}{(1-aq^{n-1})^2}.
		\end{equation}
\end{theorem}    
\begin{proof}
	Let $d=1$ in \ref{eq:9} and then divide both sides by $1-\frac{b}{a}$ so that
	\begin{equation*}
		\sum_{n=1}^{N}\left[\begin{array}{c}N\\n\end{array}\right] \frac{(q)_n\left(\frac{bq}{a}\right)_{n-1}\left(cq\right)_{n-1}(a)_{N-n}a^n}{(b)_n(cq)_n(a)_{N}}=\displaystyle\sum_{n=1}^{N} \left[\begin{array}{c}N\\n\end{array}\right]\frac{\left(\frac{b}{c}\right)_{n-1}(q)_n(cq)_{N-n} c^{n-1}}{(b)_{n-1}(cq)_N}\frac{aq^{n-1}}{(1-aq^{n-1})(1-bq^{n-1})}.
		\end{equation*}
	Now let $b=a$ and $c=1$ to arrive at \eqref{entry5fin}.
\end{proof}
Let $N \longrightarrow \infty $ in \eqref{entry5fin} to have
\begin{align*}
	\sum_{n=1}^{\infty}\frac{(q)_{n-1}a^n}{(1-q^n)(a)_n}&=\sum_{n=1}^{\infty}\frac{aq^{n-1}}{(1-aq^{n-1})^2}\nonumber\\
	&=\sum_{m=1}^{\infty}m\left(\frac{a}{q}\right)^m\sum_{n=1}^{\infty}q^{mn}\nonumber\\
	&=\sum_{m=1}^{\infty}\frac{ma^m}{1-q^m}.
	\end{align*}

\section{An identity involving a finite sum of a ${}_2\phi_{1}$}\label{fa2p1}
In this section, we derive an identity for a finite sum of a ${}_2\phi_{1}$ that is instrumental in proving all of the results in Section \ref{app}. These include, among other things, a new generalization of the generating function version of Andrews' identity for $\textup{spt}(n)$, that is, \eqref{idspt}.
We begin with a lemma.

\begin{lemma}\label{supheinelemma}
We have
\begin{align}\label{supheine}
	&\sum_{n=1}^{N}\frac{(-1)^{n-1}\left(\frac{c}{d} \right)_{n}d^n  q^{\frac{n(n+1)}{2}}}{(q)_n (q)_{N-n} (cq)_{n}}  \sum_{k=1}^n \frac{q^k}{1- q^k} \nonumber\\
	&=\frac{(\frac{c}{d})_{\infty}(dq)_{\infty}}{(q)_{N}(cq)_{\infty}(dq^{N+1})_{\infty}} \sum_{k=1}^{N}\left[\begin{array}{c}N\\k\end{array}\right]\frac{d^kq^{k(k+1)}}{(dq)_k(1-q^k)}{_2}\phi_{1}\left[\begin{array}{ccc} dq,&dq^{N+1};\; &\frac{cq^k}{d} \\dq^{k+1}& \end{array}\right] .
	\end{align}
\end{lemma}
\begin{proof}
Using van Hamme's identity 
\begin{align*}
\sum_{k=1}^{n} \frac{q^k}{1 - q^k} = \sum_{k=1}^{n}\left[\begin{matrix} n\\k\end{matrix}\right]\frac{(-1)^{k-1} q^{k(k+1)/2}}{(1-q^k)},
\end{align*}
in the first step below, we have
\begin{align}\label{befheine}
&\sum_{n=1}^{N}\frac{(-1)^{n-1}\left(\frac{c}{d} \right)_{n}d^n  q^{\frac{n(n+1)}{2}}}{(q)_n (q)_{N-n} (cq)_{n}} \sum_{k=1}^n \frac{q^k}{1- q^k}\nonumber\\
 & =   \displaystyle\sum_{n=1}^{N}\frac{(-1)^{n-1}\left(\frac{c}{d} \right)_{n}d^n  q^{\frac{n(n+1)}{2}}}{(q)_n (q)_{N-n}(cq)_{n}} \sum_{k=1}^{n}\left[\begin{matrix} n\\k\end{matrix}\right]\frac{(-1)^{k-1} q^{k(k+1)/2}}{(1-q^k)} \nonumber\\ 
& =\sum_{k=1}^{N}\frac{(-1)^{k-1} q^{k(k+1)/2}}{(q)_k (1-q^k)}\sum_{n=k}^{N}\frac{(-1)^{n-1}d^{n}\left(\frac{c}{d} \right)_{n}  q^{\frac{n(n+1)}{2}}}{(cq)_{n} (q)_{N-n}(q)_{n-k}}\nonumber\\
	&=\sum_{k=1}^{N}\frac{ q^{k(k+1)}}{(q)_k  (1-q^k)} \sum_{m=0}^{N-k} \frac{(-1)^m d^{k+m}\left(\frac{c}{d} \right)_{k+m} q^{\frac{m(m+1)}{2}+mk}}{(cq)_{k+m}(q)_{N-m-k}(q)_m }\nonumber\\
	&=\sum_{k=1}^{N}\frac{ q^{k(k+1)}}{(q)_k (q)_{N-k} (1-q^k)}\sum_{m=0}^{N-k}  \frac{\left(\frac{c}{d} \right)_{k+m}d^{k+m}\left( q^{-(N-k)} \right)_mq^{(N+1)m} }{(q)_m(cq)_{m+k} }\nonumber\\
	&=\sum_{k=1}^{N}\frac{\left(\frac{c}{d}\right)_k d^kq^{k(k+1)}}{\left(cq\right)_k(q)_k (q)_{N-k} (1-q^k)}\sum_{m=0}^{N-k}  \frac{\left(\frac{cq^{k}}{d} \right)_{m}\left( q^{-(N-k)} \right)_m(dq^{N+1})^m }{(q)_m\left(cq^{k+1}\right)_{m} },
\end{align}
where in the second-last step, we invoked  \eqref{ef} with $x=1$ and $N$ replaced by $N-k$.

The inner sum is now handled using Heine's transformation \cite[p.~359, Equations (III.1), (III.2)]{gasperrahman}, namely, for $|z|<1$ and $|\gamma|<|\beta|<1$,
\begin{equation}\label{heine}
	_{2}\phi_{1}\left[\begin{array}{ccc} \alpha,&\beta\; ;&z\\\gamma&  \end{array}\right]
=\frac{(\beta)_{\infty}(\alpha z)_{\infty}}{(\gamma)_{\infty}(z)_{\infty}}\;_{2}\phi_{1}\left[\begin{array}{ccc} \frac{\gamma}{\beta},&z\; ;&\beta\\\alpha z&  \end{array}\right].
\end{equation}
Let $\alpha=q^{-(N-k)} $, $\beta=\frac{cq^{k}}{d} $, $\gamma=cq^{k+1}$ and $z=dq^{N+1} $ in \eqref{heine} so that
\begin{align}\label{aftheine}
	&\sum_{m=0}^{N-k}  \frac{\left(\frac{cq^{k}}{d} \right)_{m}\left( q^{-(N-k)} \right)_m(dq^{N+1})^m }{(q)_m\left(cq^{k+1}\right)_{m} }
&=\frac{(dq^{k+1})_{\infty}(\frac{cq^{k}}{d})_{\infty}}{(cq^{k+1})_{\infty}(dq^{N+1})_{\infty}}\;_{2}\phi_{1}\left[\begin{array}{ccc} dq,&dq^{N+1};\; &\frac{cq^k}{d} \\dq^{k+1}& \end{array}\right].
\end{align}
Substituting \eqref{aftheine} in \eqref{befheine} results in \eqref{supheine} upon simplification.
\end{proof}
We also need the following lemma whose proof is similar to that of Corollary 4.1 of \cite{dems}. Hence we give only the outline of the proof.
\begin{lemma}\label{corlemma}
We have
	\begin{equation*}
		\sum_{n=1}^{N}\left[\begin{array}{c}N\\n\end{array}\right] \frac{\left(\frac{c}{d}\right)_n(d)^n(-1)^{n-1} q^\frac{n(n+1)}{2}}{(cq)_n} =\left(1- \frac{(dq)_N}{(c q)_N} \right).
		\end{equation*}
\end{lemma}
\begin{proof}
 Let  $z=1$ in \eqref{eq:13}. Then use the fact $(c q)_{N-n}/(cq)_N = 1/(cq^{N+1-n})_n$ so that
	\begin{align*}
	\sum_{n=1}^{N}\left[\begin{array}{c}N\\n\end{array}\right] \frac{\left(\frac{c}{d}\right)_n(d)^n(-1)^{n-1} q^\frac{n(n+1)}{2}}{(cq)_n} 
		=-\sum_{n=1}^{N}  \frac{\left(\frac{d}{c}\right)_{n}(q^{-N})_nq^n} {(q)_n\left(\frac{q^{-N}}{c}\right)_n},
	\end{align*}
Now use the  $q$-Chu-Vandermonde identity \cite[p.~354, II(6)]{gasperrahman}
\begin{equation}\label{q-Chu-Vandermonde}
		{}_{2}\phi_{1}\left[ \begin{matrix} a,  q^{-N} \\
			x \end{matrix} \, ; q, q  \right] 
		= \frac{\left(\frac{x}{a}\right)_N a^N}{(x)_N}  
		\end{equation}
with $x= q^{-N}/c$ and $a = d/c$ to write the right-hand side as a $q$-product which completes the proof.
	\end{proof}

\begin{theorem}\label{2phi1id}
The following identity holds:
	\begin{align}\label{2phi1ideqn}
	&\sum_{n=1}^{N}\frac{n(-1)^{n-1}\left(\frac{c}{d} \right)_{n}d^n  q^{\frac{n(n+1)}{2}}}{(q)_n (q)_{N-n} (cq)_{n}}  + \frac{(\frac{c}{d})_{\infty}(dq)_{\infty}}{(q)_{N}(cq)_{\infty}(dq^{N+1})_{\infty}} \sum_{k=1}^{N}\bigg[\begin{array}{c}N\\k\end{array}\bigg]\frac{d^kq^{k(k+1)} }{(dq)_k(1-q^k)}{_2}\phi_{1}\bigg(\begin{array}{ccc} dq,&dq^{N+1};\; &\frac{cq^k}{d} \\dq^{k+1}& \end{array}\bigg)\nonumber\\
	&= \frac{c}{(c-d)(q)_N}\left(1-\frac{(dq)_N}{(cq)_N}\right) +\frac{1}{(cq)_N} \sum_{k=1}^N \frac{\left(\frac{cq}{d}\right)_k (dq)_{N-k}(dq)^k}{(q)_k (q)_{N-k}(1-q^k) }.
	\end{align}
\end{theorem}
\begin{proof}
	
	Differentiate both sides of \eqref{eq:13} w.r.t $z$ and then let $z=1$ to obtain
	\begin{align}\label{Putting_z=1}
		 \sum_{n=1}^{N}\frac{(-1)^{n-1}\left(\frac{c}{d} \right)_{n}d^n  q^{\frac{n(n+1)}{2}}}{(q)_n (q)_{N-n} (cq)_{n}}  \left(n + \sum_{k=1}^n \frac{q^k}{1- q^k}  \right) =:T_1+T_2,
		\end{align}
	where
	\begin{align*}
		T_1&  =\frac{d-c}{c}\sum_{n=1}^N \frac{\left(\frac{dq}{c}\right)_{n-1} (c q)_{N-n} (cq)^n }{(q)_n (c q)_N (q)_{N-n} } \nonumber\\
		T_2&= \frac{d-c}{c}  \sum_{n=1}^N \frac{\left(\frac{dq}{c}\right)_{n-1} (c q)_{N-n} (cq)^n }{(q)_n (c q)_N (q)_{N-n} } \left(-\sum_{k=1}^{n-1} \frac{q^kd/c}{1- q^kd/c} + \sum_{k=1}^n \frac{q^k}{1-q^k} \right).
	\end{align*}
	Using \eqref{eq:13} with $z=1$ and then invoking Lemma \ref{corlemma}, we see that
	\begin{equation}\label{s1}
		T_1=\frac{1}{(q)_N}\left(1-\frac{(dq)_N}{(cq)_N}\right).
	\end{equation}
	In \cite[Corollary 3.1]{guozhang}, Guo and Zhang  have shown that if $n\geq 0$ and $0\leq m\leq n$, then
	\begin{align*}
		&\sum_{k=0\atop k\neq m}^{n}\left[\begin{matrix} n\\k\end{matrix}\right]\frac{\left(\frac{q}{x}\right)_k(x)_{n-k}}{1-q^{k-m}}x^k=(-1)^mq^{\frac{m(m+1)}{2}}\left[\begin{matrix} n\\m\end{matrix}\right](xq^{-m})_n\left(\sum_{k=0}^{n-1}\frac{xq^{k-m}}{1-xq^{k-m}}-\sum_{k=0\atop k\neq m}^{n}\frac{q^{k-m}}{1-q^{k-m}}\right).
	\end{align*}
	The $m=0$ case of the above identity gives
	\begin{align}\label{m0}
		\sum_{k=1}^{n}\frac{q^k}{1-q^k}-\sum_{k=1}^{n-1}\frac{xq^k}{1-xq^k}=\frac{x}{1-x}-\frac{1}{(x)_n}\sum_{k=1}^{n}\left[\begin{matrix} n\\k\end{matrix}\right]\frac{\left(\frac{q}{x}\right)_k(x)_{n-k}x^k}{1-q^k}.
	\end{align}
	Now invoke \eqref{m0} with $x=d/c$ to simplify the expression within the parentheses in $T_2$ to see that
	\begin{align*}
		T_2&=\frac{d-c}{c}  \sum_{n=1}^N \frac{\left(\frac{dq}{c}\right)_{n-1} (c q)_{N-n} (cq)^n }{(q)_n (c q)_N (q)_{N-n} }\left( \frac{d}{c-d} - \frac{1}{\left(\frac{d}{c}\right)_n} \sum_{k=1}^n \left[\begin{matrix} n\\k\end{matrix}\right] \frac{\left(\frac{qc}{d}\right)_k \left(\frac{d}{c}\right)_{n-k} \left(\frac{d}{c}\right)^{k}}{1- q^k}\right)\nonumber\\
		&=\frac{d}{(c-d)(q)_N}\left(1-\frac{(dq)_N}{(cq)_N}\right) +T_{2}^{*},
	\end{align*}
	where 
	\begin{equation}\label{s2s}
		T_2^{*}:=\frac{1}{(cq)_N}  \sum_{n=1}^N \frac{ \left(cq\right)_{N-n} (cq)^n }{(q)_n  (q)_{N-n}  } \sum_{k=1}^n \left[\begin{matrix} n\\k\end{matrix}\right] \frac{\left(\frac{qc}{d}\right)_k \left(\frac{d}{c}\right)_{n-k}  \left(\frac{d}{c}\right)^{k}}{1- q^k},
	\end{equation}
and in the last step, we used \eqref{s1}. Now
	\begin{align*}
		T_2^{*}& =\frac{1}{(cq)_N} \sum_{k=1}^N \frac{\left(\frac{cq}{d}\right)_k (dq)^k}{(q)_k (1-q^k) } \sum_{j=0}^{N-k} \frac{\left(\frac{d}{c}\right)_j (cq)^j (cq)_{N-j-k}}{(q)_j(q)_{N-j-k}} \nonumber \\
		& =  \frac{1}{(cq)_N} \sum_{k=1}^N \frac{\left(\frac{cq}{d}\right)_k (cq)_{N-k}(dq)^k}{(q)_k (1-q^k)(q)_{N-k} } \sum_{j=0}^{N-k} \frac{\left(\frac{d}{c}\right)_j q^j \left(q^{-(N-k)}\right)_j}{(q)_j \left(q^{-(N-k)}/c\right)_j},
	\end{align*}
	where in the last step, we used \eqref{eq:11} with $N$ replaced by $N-k$ and $n$ replaced by $j$. Now apply \eqref{q-Chu-Vandermonde} with $a = d/c$, $x= q^{-(N-k)}/c$ and $N$ replaced by $N-k$ to see that
	\begin{equation}\label{appl_Chu}
		\sum_{j=0}^{N-k} \frac{\left(\frac{d}{c}\right)_j q^j \left(q^{-(N-k)}\right)_j}{(q)_j \left(q^{-(N-k)}/c\right)_j}=\frac{(dq)_{N-k}}{(cq)_{N-k}}.
	\end{equation}
 From  \eqref{appl_Chu} and \eqref{s2s},
 \begin{align*}
T_2^{*}=\frac{1}{(cq)_N} \sum_{k=1}^N \frac{\left(\frac{cq}{d}\right)_k (dq)^k(dq)_{N-k}}{(q)_k (1-q^k)(q)_{N-k} },
 \end{align*}
which, along with \eqref{s1}, implies
	\begin{align}\label{final_second term}
		T_2= \frac{d}{(c-d)(q)_N}\left(1-\frac{(dq)_N}{(cq)_N}\right) +\frac{1}{(cq)_N} \sum_{k=1}^N \frac{\left(\frac{cq}{d}\right)_k (dq)^k(dq)_{N-k}}{(q)_k (1-q^k)(q)_{N-k} }.
	\end{align}
	Hence from \eqref{Putting_z=1}, \eqref{s1} and \eqref{final_second term},
	\begin{align*}
		& \sum_{n=1}^{N}\frac{(-1)^{n-1}\left(\frac{c}{d} \right)_{n}d^n  q^{\frac{n(n+1)}{2}}}{(q)_n (q)_{N-n} (cq)_{n}}  \left(n + \sum_{k=1}^n \frac{q^k}{1- q^k}  \right) \\ 
		&=\frac{c}{(c-d)(q)_N}\left(1-\frac{(dq)_N}{(cq)_N}\right) +\frac{1}{(cq)_N} \sum_{k=1}^N \frac{\left(\frac{cq}{d}\right)_k (dq)^k(dq)_{N-k}}{(q)_k (1-q^k)(q)_{N-k} }
	\end{align*}
Finally from the above equation and Lemma \ref{supheinelemma}, we arrive at \eqref{2phi1ideqn}.
\end{proof}

\section{Applications of Theorem \ref{2phi1id}}\label{app}
There are several applications of Theorem \ref{2phi1id}. We begin with the most appealing one generalizing  the generating function version of Andrews' famous identity 
\begin{equation} \label{idspt}
	\textup{spt}(n)=np(n)-\frac{1}{2}N_2(n).
\end{equation}

\subsection{A generalization of Andrews' identity}

\begin{theorem}\label{sptgenthm}
	We have
	\begin{align}\label{sptgen}
		\frac{1	}{(q)_{\infty}}\sum_{n=1}^{\infty}\frac{n(-d)^{n-1}(q/d)_{n-1}q^{n(n+1)/2}}{(q)_{n}^2}
		&=\frac{1}{(q)_{\infty}}\sum_{n=1}^{\infty}\frac{nq^n(dq)_{n-1}}{(q)_n}\nonumber\\
		&\quad-\frac{(dq)_{\infty}}{(q)_{\infty}}\sum_{j=1}^{\infty}\frac{q^{j^2}}{(q)_{j}^{2}}\sum_{n=1}^{j}\frac{q^n}{(1-dq^n)(1-q^n)}.
	\end{align}	
\end{theorem}
\begin{proof}
Let $c=1$ in Theorem \ref{2phi1id} and then divide both sides by $d-1$ to obtain after simplification
	\begin{align}\label{befmmain}
	&\sum_{n=1}^{N}\frac{n(-d)^{n-1}\left(\frac{q}{d} \right)_{n-1}q^{\frac{n(n+1)}{2}}}{(q)_n^2 (q)_{N-n}}  + \frac{(\frac{q}{d})_{\infty}(dq)_{\infty}}{d(q)_{N}(dq^{N+1})_{\infty}(q)_{\infty}} \sum_{k=1}^{N}\bigg[\begin{array}{c}N\\k\end{array}\bigg]\frac{d^kq^{k(k+1)} }{(dq)_k(1-q^k)}{_2}\phi_{1}\bigg(\begin{array}{ccc} dq,&dq^{N+1};\; &\frac{q^k}{d} \\dq^{k+1}& \end{array}\bigg)\nonumber\\
	&= \frac{1}{(1-d)^2(q)_N}\left(\frac{(dq)_N}{(q)_N}-1\right) +\frac{1}{(d-1)(q)_N} \sum_{k=1}^N \frac{\left(\frac{q}{d}\right)_k (dq)_{N-k}(dq)^k}{(q)_k (q)_{N-k}(1-q^k) }.
\end{align}
Letting $N\to\infty$ gives
	\begin{align}\label{mmain}
	&\frac{1}{(q)_{\infty}}\sum_{n=1}^{\infty}\frac{n(-d)^{n-1}\left(\frac{q}{d} \right)_{n-1}q^{\frac{n(n+1)}{2}}}{(q)_n^2}  + \frac{(\frac{q}{d})_{\infty}(dq)_{\infty}}{d(q)_{\infty}^{2}} \sum_{k=1}^{\infty}\frac{d^kq^{k(k+1)} }{(q)_{k}(dq)_k(1-q^k)}\sum_{j=1}^{\infty}\frac{(dq)_{j}(q^{k}/d)^{j}}{(dq^{k+1})_j(q)_j}\nonumber\\
	&= \frac{1}{(1-d)^2(q)_{\infty}}\left(\frac{(dq)_{\infty}}{(q)_{\infty}}-1\right) +\frac{(dq)_{\infty}}{(d-1)(q)_{\infty}^2} \sum_{k=1}^{\infty} \frac{\left(\frac{q}{d}\right)_k (dq)^k}{(q)_k (1-q^k) }.
\end{align}
Now using Jackson's transformation \cite[p.~526]{aar}
\begin{equation}\label{jackson}
	\sum_{n=0}^{\infty}\frac{(\alpha)_n(\beta)_n}{(\gamma)_n(q)_n}z^n=\frac{(\alpha z)_{\infty}}{(z)_{\infty}}\sum_{n=0}^{\infty}\frac{(\alpha)_n(\gamma/\beta)_n(-\beta z)^nq^{n(n-1)/2}}{(\gamma)_n(\alpha z)_n(q)_n}
\end{equation}
with $\alpha=dq, \gamma=dq^{k+1}, \beta=0$ and $z=q^k/d$, we see that the inner sum in \eqref{mmain} transforms as
\begin{align*}
	\sum_{j=0}^{\infty}\frac{(dq)_{j}(q^{k}/d)^{j}}{(dq^{k+1})_j(q)_j}=\frac{(q^{k+1})_{\infty}}{(q^{k}/d)_{\infty}}\sum_{j=0}^{\infty}\frac{(dq)_{j}q^{j^2+2jk}}{(dq^{k+1})_j(q^{k+1})_j(q)_j}.
	\end{align*}
 Thus, the second expression on the left-hand side of \eqref{mmain} can be written as
 \begin{align*}
 	&\frac{(\frac{q}{d})_{\infty}(dq)_{\infty}}{d(q)_{\infty}^{2}} \sum_{k=1}^{\infty}\frac{d^kq^{k(k+1)} }{(q)_{k}(dq)_k(1-q^k)}\sum_{j=0}^{\infty}\frac{(dq)_{j}(q^{k}/d)^{j}}{(dq^{k+1})_j(q)_j}\nonumber\\
 	&=	\frac{(\frac{q}{d})_{\infty}(dq)_{\infty}}{d(q)_{\infty}^{2}} \sum_{k=1}^{\infty}\frac{d^kq^{k(k+1)} }{(q)_{k}(dq)_k(1-q^k)}\frac{(q^{k+1})_{\infty}}{(q^{k}/d)_{\infty}}\sum_{j=0}^{\infty}\frac{(dq)_{j}q^{j^2+2jk}}{(dq^{k+1})_j(q^{k+1})_j(q)_j}\nonumber\\
 	&=\frac{(dq)_{\infty}}{d(q)_{\infty}}\sum_{k=1}^{\infty}\frac{(dq)^k(q/d)_{k-1}}{(q)_k(1-q^k)}\sum_{j=0}^{\infty}\frac{(dq)_{j}q^{(j+k)^2}}{(dq)_{j+k}(q)_{j+k}(q)_j}\nonumber\\
 		&=\frac{(dq)_{\infty}}{d(q)_{\infty}}\sum_{k=1}^{\infty}\frac{(dq)^k(q/d)_{k-1}}{(q)_k(1-q^k)}\sum_{j=k}^{\infty}\frac{(dq)_{j-k}q^{j^2}}{(dq)_{j}(q)_{j}(q)_{j-k}}\nonumber\\
 		&=\frac{(dq)_{\infty}}{d(q)_{\infty}}\sum_{j=1}^{\infty}\frac{q^{j^2}}{(dq)_j(q)_j}\sum_{k=1}^{j}\frac{(dq)^k(q/d)_{k-1}(dq)_{j-k}}{(q)_k(q)_{j-k}(1-q^k)}.
 \end{align*}
Next, using \eqref{m0} with $x=dq$ and simplifying, we find that
\begin{align}\label{m0app}
	\sum_{k=1}^{j}\frac{(dq)^k(q/d)_{k-1}(dq)_{j-k}}{(q)_k(q)_{j-k}(1-q^k)}=\frac{(dq)_j}{(1-1/d)(q)_j}\left(\sum_{k=1}^{j}\frac{dq^k}{1-dq^k}-\sum_{k=1}^{j}\frac{q^k}{1-q^k}\right).
\end{align}
Substituting \eqref{m0app} in the last expression in \eqref{m0bef}, the second expression on the left-hand side of \eqref{mmain} finally simplifies to
 \begin{align}\label{m0bef}
	\frac{(\frac{q}{d})_{\infty}(dq)_{\infty}}{d(q)_{\infty}^{2}} \sum_{k=1}^{\infty}\frac{d^kq^{k(k+1)} }{(q)_{k}(dq)_k(1-q^k)}\sum_{j=1}^{\infty}\frac{(dq)_{j}(q^{k}/d)^{j}}{(dq^{k+1})_j(q)_j}&=\frac{(dq)_{\infty}}{(d-1)(q)_{\infty}}\sum_{j=1}^{\infty}\frac{q^{j^2}}{(q)_j^{2}}\sum_{k=1}^{j}\left(\frac{dq^k}{1-dq^k}-\frac{q^k}{1-q^k}\right)\nonumber\\
	&=\frac{(dq)_{\infty}}{(q)_{\infty}}\sum_{j=1}^{\infty}\frac{q^{j^2}}{(q)_j^{2}}\sum_{k=1}^{j}\frac{q^k}{(1-dq^k)(1-q^k)}.
	\end{align}
Next, let $F_d(q)$ denote the expression on the right-hand side of \eqref{mmain}. Using the identity \cite[Corollary 2.4]{dixitmaji18}, namely
\begin{align*}
	\sum_{k=1}^{\infty}\frac{d^{k-1}(q/d)_{k-1} q^k}{(q)_k}=\frac{1}{1-d} \left(1- \frac{(q)_{\infty} }{(d q)_\infty }  \right),
\end{align*}
in the second step below, we see that
\begin{align}\label{m0last0}
	F_d(q)&=\frac{(dq)_{\infty}}{(d-1)(q)_{\infty}^{2}}\left\{\frac{1}{(1-d)}\left(\frac{(q)_{\infty}}{(dq)_{\infty}}-1\right)+\sum_{k=1}^{\infty} \frac{\left(\frac{q}{d}\right)_k (dq)^k}{(q)_k (1-q^k) }\right\}\nonumber\\
	&=\frac{(dq)_{\infty}}{(d-1)(q)_{\infty}^{2}}\left\{-\frac{1}{d}\sum_{k=1}^{\infty}\frac{(q/d)_{k-1} (dq)^k}{(q)_k}+\sum_{k=1}^{\infty} \frac{\left(\frac{q}{d}\right)_k (dq)^k}{(q)_k (1-q^k) }\right\}\nonumber\\
	&=\frac{(dq)_{\infty}}{(d-1)(q)_{\infty}^{2}}\sum_{k=1}^{\infty}\frac{(q/d)_{k-1} (dq)^k}{(q)_k}\left(-\frac{1}{d}+\frac{1-q^k/d}{1-q^k}\right)\nonumber\\
	&=\frac{(dq)_{\infty}}{d(q)_{\infty}^{2}}\sum_{k=1}^{\infty}\frac{(q/d)_{k-1} (dq)^k}{(q)_k(1-q^k)}\nonumber\\
	&=\frac{(dq)_{\infty}}{(q)_{\infty}^2}\sum_{n=1}^{\infty}\frac{q^n}{(1-dq^n)(1-q^n)},
\end{align} 
where in the last step we used the special case $z=1$ of the identity
\begin{align*}
	\sum_{n=1}^{\infty}   \frac{(q)_{n-1} z^nq^n }{ (1-d q^n) (zq)_n } &=z\sum_{n=1}^{\infty}\frac{(zq/d)_{n-1}}{(zq)_n}\frac{d^{n-1}q^n}{1-zq^{n}}
\end{align*}
which is valid for $|zq|<1$ and $|dq|<1$. 

For $|q|<1$ and $d\neq d^{-m}, m\in\mathbb{N}$, Andrews \cite[p.~159]{ms_problem} has generalized Uchimura's identity \cite[Theorem 2]{uchimura} as follows:
\begin{align}\label{andrewsgen}
\sum_{n=1}^{\infty}\frac{q^n}{(1-dq^n)(1-q^n)}=\sum_{n=1}^{\infty}\frac{nq^n(q^{n+1})_{\infty}}{(dq^n)_{\infty}}.
\end{align}
Hence substituting \eqref{andrewsgen} in \eqref{m0last0} leads to
\begin{align}\label{m0last}
	F_d(q)=\frac{1}{(q)_{\infty}}\sum_{n=1}^{\infty}\frac{nq^n(dq)_{n-1}}{(q)_n}.
	\end{align}
Substituting \eqref{m0bef} and \eqref{m0last} in \eqref{mmain} and rearranging, we arrive at \eqref{sptgen}.
\end{proof}

\begin{corollary}
The identity in \eqref{idspt} holds.
\end{corollary}
\begin{proof}
	Let $d=1$ in \eqref{sptgen}. This gives
		\begin{align}\label{mmaind1}
		\frac{1}{(q)_{\infty}}\sum_{n=1}^{\infty}\frac{n(-1)^{n-1}q^{\frac{n(n+1)}{2}}}{(q)_n(1-q^n)}  &=\frac{1}{(q)_{\infty}}\sum_{n=1}^{\infty}\frac{q^n}{(1-q^n)^2}-\sum_{j=1}^{\infty}\frac{q^{j^2}}{(q)_{j}^{2}}\sum_{n=1}^{j}\frac{q^n}{(1-q^n)^2}.
	\end{align}
From \cite[Theorem 3.8]{agl13},
\begin{equation}\label{1}
	\sum_{n=1}^{\infty}\textup{spt}(n)q^n=	\frac{1}{(q)_{\infty}}\sum_{n=1}^{\infty}\frac{n(-1)^{n-1}q^{\frac{n(n+1)}{2}}}{(q)_n(1-q^n)} ,
\end{equation}
whereas, from \cite[Equation (3.3)]{andrews08},
\begin{align}\label{2}
	\sum_{n=1}^{\infty}np(n)q^n&=\frac{1}{(q)_{\infty}}\sum_{n=1}^{\infty}\frac{nq^n}{1-q^n}.
\end{align}
Moreover, from \cite[Equation (7.14)]{dixitmaji18},
\begin{align}\label{3}
	\sum_{j=1}^{\infty}\frac{q^{j^2}}{(q)_{j}^{2}}\sum_{n=1}^{j}\frac{q^n}{(1-q^n)^2}=\frac{1}{2}\left.\frac{d^2}{dz^2}\sum_{j=0}^{\infty}\frac{q^{j^2}}{(zq)_j(z^{-1}q)_j}\right|_{z=1}.
\end{align}
From \eqref{mmaind1}, \eqref{1}, \eqref{2} and \eqref{3}, we arrive at 
\begin{align*}
\sum_{n=1}^{\infty}\textup{spt}(n)q^n=\sum_{n=1}^{\infty}np(n)q^n-\frac{1}{2}\sum_{n=1}^{\infty}N_2(n)q^n,
\end{align*}
which establishes \eqref{idspt} upon comparing the coefficients of $q^n$ on both sides.
\end{proof}

\begin{corollary}
We have
\begin{align*}
\sum_{n=1}^{\infty}\frac{n(-q)_{n-1}q^{n(n+1)/2}}{(q)_n^{2}}=\sum_{n=1}^{\infty}\frac{nq^n(-q)_{n-1}}{(q)_n}-(-q)_{\infty}\sum_{j=1}^{\infty}\frac{q^{j^2}}{(q)_{j}^{2}}\sum_{n=1}^{j}\frac{q^n}{(1-q^{2n})}.
\end{align*}
\end{corollary}
\begin{proof}
Let $d=-1$ in \eqref{sptgen} and multiply the resulting identity by $(q)_{\infty}$.
\end{proof}
\begin{remark}
We observe that $\displaystyle\sum_{n=1}^{\infty}\frac{nq^n(-q)_{n-1}}{(q)_n}$ is the generating function of the sum of largest parts (counted with multiplicty $1$) in those overpartitions of a positive integer whose largest part is always overlined. For example, there are seven overpartitions of $4$ whose largest part is always overlined, namely, $\overline{4}, \overline{3}+1, \overline{3}+\overline{1}, \overline{2}+\overline{2}, \overline{2}+\overline{1}+1, \overline{2}+1+1$ and $\overline{1}+\overline{1}+\overline{1}+\overline{1}$. Then the coefficient $q^4$ in the series expansion of $\displaystyle\sum_{n=1}^{\infty}\frac{nq^n(-q)_{n-1}}{(q)_n}$  is 4+3+3+2+2+2+1=17.
\end{remark}

\subsection{A generalization of an identity for the generating function of $N_{\textup{SC}}(n)$ }

Let $V$ be the set of vector partitions, that is, $V=\mathcal{D}\times\mathcal{P}\times\mathcal{P}$, where $\mathcal{P}$ denotes the set of unrestricted partitions and $\mathcal{D}$ denotes the set of partitions into distinct parts. Define $S$ to be the set of vector partitions given below:
\begin{equation*}
	S:=\{\vec{\pi}=(\pi_1, \pi_2, \pi_3)\in V: 1\leq s(\pi_1)<\infty\hspace{1mm}\text{and}\hspace{1mm}s(\pi_1)\leq\min(s(\pi_2), s(\pi_3))\}.
\end{equation*}
Let $\omega_1(\vec{\pi})=(-1)^{\#(\pi_1)-1}$. Let $\imath: S\to S$ be the involution map defined by
\begin{equation*}
	\imath(\vec{\pi})=\imath(\pi_1, \pi_2, \pi_3)=\imath(\pi_1, \pi_3, \pi_2).
\end{equation*}
The partitions $\vec{\pi}=(\pi_1, \pi_2, \pi_3)$ from the set $S$  are simply called $S$-partitions. Define an $S$-partition $\vec{\pi}=(\pi_1, \pi_2, \pi_3)$  to be self-conjugate $S$-partition if it is a fixed point of $\imath$, that is, if and only if $\pi_2=\pi_3$. Moreover, let $N_{\textup{SC}}(n)$ denote the number of self-conjugate $S$-partitions counted according to the weight $\omega_1$, that is,
\begin{equation*}
	N_{\textup{SC}}(n)=\sum_{\vec{\pi}\in S, |\vec{\pi}|=n \atop \imath(\vec{\pi})=\vec{\pi}}\omega_1(\vec{\pi}).
\end{equation*}
Andrews, Garvan and Liang \cite[Theorem 3.8, Equation (3.25)]{agl13} showed that
\begin{align*}
	\sum_{n=1}^{\infty}N_{\textup{SC}}(n)q^n=\frac{1}{(q)_{\infty}}\sum_{n=1}^{\infty}\frac{n(-1)^{n-1}q^{n(n+1)/2}}{(q)_n(1+q^n)}.
\end{align*}
In \cite[Corollary 2.12]{dixitmaji18}, the following result for the generating function of $N_{\textup{SC}}(n)$ was proved:
\begin{align}\label{nsceqn}
	&(q)_{\infty}\sum_{n=1}^{\infty}N_{\textup{SC}}(n)q^n+\frac{1}{2}\frac{(q)_{\infty}}{(-q)_{\infty}}\sum_{n=1}^{\infty}\frac{q^{\frac{n(n+1)}{2}}}{(1-q^n)(q)_n}\left(\frac{(-q)_n}{(q)_n}-1\right)\nonumber\\
	&=\frac{1}{4}-\frac{1}{4}\frac{(q)_{\infty}}{(-q)_{\infty}}+\frac{1}{2}\frac{(q)_{\infty}}{(-q)_{\infty}}\sum_{n=1}^{\infty}\frac{(-q)_n}{(q)_n}\frac{q^n}{1-q^n}.
\end{align}
In what follows, we generalize the above identity by means of an extra parameter $d$.
\begin{theorem}
We have
	\begin{align}\label{nscgen}
&\sum_{n=1}^{\infty}\frac{n(-1)^{n-1}\left(\frac{-1}{d}\right)_{n}d^nq^{n(n+1)/2}}{(q^2;q^2)_n}+\frac{(\frac{-1}{d})_{\infty}(dq)_{\infty}}{(-q)_{\infty}} \sum_{k=1}^{\infty}\frac{d^kq^{k(k+1)}}{(q)_k(dq)_k(1-q^k)}\sum_{n=0}^{\infty} \frac{(dq)_{n}}{(dq^{k+1} )_{n}(q)_{n}}\left(\frac{-q^{k}}{d} \right)^{n}\nonumber\\
&=\frac{1}{(1+d)}\left(1-\frac{(dq)_{\infty}}{(-q)_{\infty}}\right)+\frac{(dq)_{\infty}}{(-q)_{\infty}}\sum_{n=1}^{\infty}\frac{(-q/d)_n(dq)^n}{(q)_n(1-q^n)}.
	\end{align}
\end{theorem}
\begin{proof}
	Let $c=-1$ in Theorem \ref{2phi1id} and then let $N\to\infty$.
\end{proof}
\begin{corollary}
	Identity \eqref{nsceqn} holds.
\end{corollary}
\begin{proof}
	Let $d=1$ in \eqref{nscgen}. This results in
	\begin{align}\label{aux0}
	&\sum_{n=1}^{\infty}\frac{n(-1)^{n-1}q^{n(n+1)/2}}{(q)_n(1+q^n)}+(q)_{\infty}\sum_{k=1}^{\infty}\frac{q^{k(k+1)}}{(q)_k^{2}(1-q^k)}F(0;q^k;-q^k)\nonumber\\
	&=\frac{1}{4}\left(1-\frac{(q)_{\infty}}{(-q)_{\infty}}\right)+\frac{1}{2}\frac{(q)_{\infty}}{(-q)_{\infty}}\sum_{n=1}^{\infty}\frac{(-q)_n}{(q)_n}\frac{q^n}{1-q^n}.
	\end{align}
But from \cite[Equation (7.26)]{dixitmaji18}, we have
\begin{align}\label{aux}
	(q)_{\infty}\sum_{k=1}^{\infty}\frac{q^{k(k+1)}}{(q)_k^{2}(1-q^k)}F(0;q^k;-q^k)=\frac{1}{2}\frac{(q)_{\infty}}{(-q)_{\infty}}\sum_{k=1}^{\infty}\frac{q^{k(k+1)/2}}{(q)_k(1-q^k)}\left(\frac{(-q)_k}{(q)_k}-1\right).
\end{align}
Substituting \eqref{aux} in \eqref{aux0}, we arrive at \eqref{nsceqn}.
\end{proof}
\begin{remark}
If we let $d=-1$ in \eqref{nscgen}, we simply obtain the \eqref{entry4} with $a=-1$.
\end{remark}
\subsection{Other special cases}
\begin{theorem}
We have
\begin{align}\label{c=0}
	\frac{1}{(dq)_{\infty}}\sum_{n=1}^{\infty}\frac{n(-1)^{n-1}d^nq^{n(n+1)/2}}{(q)_n}+\sum_{n=1}^{\infty}\frac{d^nq^{n(n+1)}}{(q)_n(dq)_n(1-q^n)}=\sum_{n=1}^{\infty}\frac{(dq)^n}{(q)_n(1-q^n)}.
\end{align}
\end{theorem}
\begin{proof}
	Let $c=0$ in Theorem \ref{2phi1id} to obtain
	\begin{align*}
		\sum_{n=1}^{N}\left[\begin{array}{c}N\\n\end{array}\right]n(-1)^{n-1}d^nq^{n(n+1)/2}+(dq)_N\sum_{n=1}^{N}\left[\begin{array}{c}N\\n\end{array}\right]\frac{d^nq^{n(n+1)}}{(dq)_n(1-q^n)}=\sum_{n=1}^{N}\left[\begin{array}{c}N\\n\end{array}\right]\frac{(dq)^n(dq)_{N-n}}{1-q^n}.
	\end{align*}
Now let $N\to\infty$ in the above identity to arrive at \eqref{c=0}.
\end{proof}
\begin{corollary}
We have
\begin{align*}
	\frac{1}{(q)_{\infty}}\sum_{n=1}^{\infty}\frac{n(-1)^{n-1}q^{n(n+1)/2}}{(q)_n}+\sum_{n=1}^{\infty}\frac{q^{n(n+1)}}{(q)_n^{2}(1-q^n)}=\sum_{n=1}^{\infty}\frac{q^n}{(q)_n(1-q^n)}.
\end{align*}
\end{corollary}
\begin{proof}
	Let $d=1$ in \eqref{c=0}.
\end{proof}
\begin{remark}
	Comparing with \cite[Corollary 2.10]{dixitmaji18}, we readily see that
	\begin{equation*}
\frac{1}{(q)_{\infty}}\sum_{n=1}^{\infty}\frac{n(-1)^{n-1}q^{n(n+1)/2}}{(q)_n}=\sum_{n=1}^{\infty}\frac{q^n}{1-q^n}.
	\end{equation*}
\end{remark}
\begin{corollary}
We have
\begin{align*}
\frac{-1}{(-q)_{\infty}}\sum_{n=1}^{\infty}\frac{nq^{n(n+1)/2}}{(q)_n}+\sum_{n=1}^{\infty}\frac{(-1)^nq^{n(n+1)}}{(q^2;q^2)_n(1-q^n)}=\sum_{n=1}^{\infty}\frac{(-q)^n}{(q)_n(1-q^n)}.
\end{align*}
\end{corollary}
\begin{proof}
	Let $d=-1$ in \eqref{c=0}.
\end{proof}
\begin{remark}
	Note that $\displaystyle\sum_{n=1}^{\infty}\frac{nq^{n(n+1)/2}}{(q)_n}$ is the generating function of the number of parts in all partitions of a positive integer into distinct parts. 
\end{remark}
\begin{theorem}
The following identity is valid:
	\begin{align}\label{d tends to zero}
		\sum_{n=0}^{\infty}\frac{nc^nq^{n^2}}{(q)_n(cq)_n}-\sum_{k=1}^{\infty}\frac{(-c)^kq^{\frac{k(k+1)}{2}}}{(q)_k(1-q^k)}\sum_{j=0}^{\infty}\frac{c^jq^{(j+k)^2}}{(cq)_{j+k}(q)_j}=\frac{1}{(cq)_{\infty}}-1-\frac{1}{(cq)_{\infty}}\sum_{k=1}^{\infty}\frac{(-c)^kq^{\frac{k(k+3)}{2}}}{(q)_k(1-q^k)}.
				\end{align}
\end{theorem}
\begin{proof}
Let $N\to\infty$ in Theorem \ref{2phi1id} thereby obtaining
\begin{align}\label{interm0}
	&\frac{1}{(q)_\infty}\sum_{n=1}^{\infty}\frac{n(-1)^{n-1}\left(\frac{c}{d} \right)_{n}d^n  q^{\frac{n(n+1)}{2}}}{(q)_n(cq)_{n}}  + \frac{(\frac{c}{d})_{\infty}(dq)_{\infty}}{(q)_{\infty}(cq)_{\infty}} \sum_{k=1}^{\infty}\frac{d^kq^{k(k+1)} }{(dq)_k(q)_k(1-q^k)}\sum_{j=0}^{\infty}\frac{(dq)_j(cq^k/d)^j}{(dq^{k+1})_j(q)_j}\nonumber\\
	&= \frac{c}{(c-d)(q)_\infty}\left(1-\frac{(dq)_\infty}{(cq)_\infty}\right) +\frac{(dq)_\infty}{(cq)_\infty(q)_\infty} \sum_{k=1}^\infty \frac{\left(\frac{cq}{d}\right)_k (dq)^k}{(q)_k (1-q^k) }.
\end{align}
We now want to let $d\to0$. However, before we do that, we need to transform the double sum on the left-hand side into a suitable form. To that end, we employ Jackson's transformation \eqref{jackson} with $\alpha=dq, \gamma=dq^{k+1}, \beta=0$ and $z=cq^k/d$ so that
\begin{align}\label{interm}
	\sum_{j=0}^{\infty}\frac{(dq)_{j}(cq^{k}/d)^{j}}{(dq^{k+1})_j(q)_j}=\frac{(cq^{k+1})_{\infty}}{(cq^{k}/d)_{\infty}}\sum_{j=0}^{\infty}\frac{(dq)_{j}q^{j^2+2jk}c^j}{(dq^{k+1})_j(cq^{k+1})_j(q)_j}.
\end{align}
Substituting \eqref{interm} in the double sum on the left-hand side of \eqref{interm0} and simplifying as in \eqref{m0bef}, we see that
\begin{align}\label{interm1}
	\frac{(\frac{c}{d})_{\infty}(dq)_{\infty}}{(q)_{\infty}(cq)_{\infty}} \sum_{k=1}^{\infty}\frac{d^kq^{k(k+1)} }{(dq)_k(q)_k(1-q^k)}\sum_{j=0}^{\infty}\frac{(dq)_j(cq^k/d)^j}{(dq^{k+1})_j(q)_j}=\frac{(dq)_{\infty}}{(q)_{\infty}}\sum_{k=1}^{\infty}\frac{(c/d)_k(dq)^k}{(q)_k(1-q^k)}\sum_{j=0}^{\infty}\frac{(dq)_jc^jq^{(j+k)^2}}{(dq)_{j+k}(cq)_{j+k}(q)_j}.
\end{align}
Substitute \eqref{interm1} in \eqref{interm0} and then let $d\to0$ to finally deduce \eqref{d tends to zero}.
\end{proof}
\begin{remark}
The identity in \eqref{interm0} is equivalent to \cite[Equation (4.5)]{bem}.
\end{remark}

\section{A finite analogue of the $\textup{ospt}$-function of Andrews, Chan and Kim}

In \cite[p. 78]{andrewschankim}, Andrews, Chan and Kim considered the odd moments of rank and crank, namely, 
\begin{equation*}
	\overline{N}_j(n):=\sum_{k=1}^{\infty}k^jN(k, n)\hspace{3mm}\text{and}\hspace{3mm}	\overline{M}_j(n):=\sum_{k=1}^{\infty}k^jM(k, n),
\end{equation*}	
where $N(k, n)$ (resp. $M(k, n)$) denote the number of partitions of $n$ with rank (resp. crank) $k$.

They further defined
\begin{align*}
	C_1(q)&:=\sum_{n=1}^{\infty}\overline{M}_1(n)q^n,\\
	R_1(q)&:=\sum_{n=1}^{\infty}\overline{N}_1(n)q^n,
\end{align*}
and showed that \cite[Theorems 1, 2]{andrewschankim}
\begin{align}
	C_1(q)&=\frac{1}{(q)_{\infty}}\sum_{n=1}^{\infty}\frac{(-1)^{n+1}q^{n(n+1)/2}}{1-q^n}=\sum_{k=0}^{\infty}\frac{kq^{k^2}}{(q)_k^{2}},\label{c1q}\\
	R_1(q)&=\frac{1}{(q)_{\infty}}\sum_{n=1}^{\infty}\frac{(-1)^{n+1}q^{n(3n+1)/2}}{1-q^n}.\label{r1q}
\end{align}
Their work culminated into an important inequality, namely, $\overline{M}_1(n)>\overline{N}_1(n)$ for all positive integers $n$, which, in fact, led them to consider the odd spt-function $\textup{ospt}(n)$ defined by
\begin{equation*}
\textup{ospt}(n)=\overline{M}_1(n)-\overline{N}_1(n).
\end{equation*}  
They then interpreted it combinatorially in terms of even and odd strings in the partitions of $n$. See \cite[p.~80]{andrewschankim} for the definitions. Moreover, they extended all these results for higher values of $k>1$ as well. 

In this section, we are concerned with the finite analogues of the first odd moments of rank and crank.

We first note that the finite analogues of  ranks and cranks, and of $N(k, n), M(k, n)$ as well as of the rank and crank moments were considered in \cite[p.~8-10]{dems} .  The finite analogues of $N(k, n)$ and $M(k, n)$, respectively, $N_{S_{1}}(k, n)$ and $M_{S_{1}}(k, n)$, satisfy 
\begin{equation*}
	N_{S_{1}}(-k, n)=N_{S_{1}}(k, n)\hspace{1mm}\text{and}\hspace{1mm} 	M_{S_{1}}(-k, n)=M_{S_{1}}(k, n),
\end{equation*}
because of which one needs to consider the following modified finite rank and crank moments for getting nontrivial odd moments:
 \begin{equation*}
 	\overline{N}_j(n, N):=\sum_{k=1}^{\infty}k^jN_{S_{1}}(k, n)\hspace{3mm}\text{and}\hspace{3mm}	\overline{M}_j(n, N):=\sum_{k=1}^{\infty}k^jM_{S_{1}}(k, n).
 \end{equation*}	
Their generating functions can then be defined to be
\begin{equation*}
	R_{j}(q, N):=\sum_{n=1}^{\infty}\overline{N}_j(n, N)q^n\hspace{3mm}\text{and}\hspace{3mm}C_{j}(q, N):=\sum_{n=1}^{\infty}\overline{M}_j(n, N)q^n.
\end{equation*}
 Here we are concerned with $R_{1}(q, N)$ and $C_{1}(q, N)$. Note that from \cite[p.~252, Theorem 4.1]{andpar},
 \begin{align*}
 	\frac{(q)_{N}}{(zq)_{N}(z^{-1}q)_{N}}=\frac{1}{(q)_N}+(1-z)\sum_{n=1}^{N}\left[\begin{matrix} N \\ n \end{matrix}\right]\frac{(-1)^n(q)_nq^{n(n+1)/2}}{(q)_{n+N}}\left(\frac{1}{1-zq^n}-\frac{1}{z-q^n}\right),
 \end{align*}
where the left-hand side is the finite analogue of the crank generating function.

 Applying the differential operator $z\frac{\partial}{\partial z}$ on both sides, we get
  \begin{align*}
 z\frac{\partial}{\partial z}		\frac{(q)_{N}}{(zq)_{N}(z^{-1}q)_{N}}=z\sum_{n=1}^{N}\left[\begin{matrix} N \\ n \end{matrix}\right]\frac{(-1)^{n+1}(q)_nq^{n(n+1)/2}}{(q)_{n+N}}\left(\frac{1-q^n}{(1-zq^n)^2}-\frac{1-q^n}{(z-q^n)^2}\right).
\end{align*}
Now since only the first sum on the right-hand side contributes to positive powers of $z$ when expanding the right-hand side as a Laurent series in $z$, we deduce that
\begin{align}\label{c1qnexpr}
	C_1(q, N)&=\lim_{z\to1}z\sum_{n=1}^{N}\left[\begin{matrix} N \\ n \end{matrix}\right]\frac{(-1)^{n+1}(q)_nq^{n(n+1)/2}}{(q)_{n+N}}\frac{1-q^n}{(1-zq^n)^2}\nonumber\\
	&=\sum_{n=1}^{N}\left[\begin{matrix} N \\ n \end{matrix}\right]\frac{(-1)^{n+1}(q)_nq^{n(n+1)/2}}{(q)_{n+N}(1-q^n)}.
\end{align} 
It is easy to see that letting $N\to\infty$ in the above identity leads to the first equality in \eqref{c1q}.

We now proceed towards obtaining a finite analogue of the first odd rank moment.Moreover, from \cite[p.~252, Theorem 2.1]{andpar}, we have
\begin{align*}
	\sum_{n=0}^{N}\left[\begin{matrix} N \\ n \end{matrix}\right]\frac{(q)_nq^{n^2}}{(zq)_n(z^{-1}q)_n}=\frac{1}{(q)_N}+(1-z)\sum_{n=1}^{N}\left[\begin{matrix} N \\ n \end{matrix}\right]\frac{(-1)^n(q)_nq^{n(3n+1)/2}}{(q)_{n+N}}\left(\frac{1}{1-zq^n}-\frac{1}{z-q^n}\right),
\end{align*}
where the left-hand side is a finite analogue of the rank generating function as given in \cite[Theorem 2.2]{dems}. Hence proceeding similarly as in the derivation of \eqref{c1qnexpr}, we arrive at
 \begin{align}\label{r1qnexpr}
	R_1(q, N)
	=\sum_{n=1}^{N}\left[\begin{array}{c}N\\n\end{array}\right]\frac{(-1)^{n+1}(q)_nq^{n(3n+1)/2}}{(q)_{n+N}(1-q^n)},
\end{align}
which indeed gives us \eqref{r1q} upon letting $N\to\infty$.

Next, from \eqref{c1qnexpr} and \eqref{r1qnexpr}, 
\begin{align}\label{difference}
	C_1(q, N)-R_1(q, N)=\sum_{n=1}^{N}\left[\begin{array}{c}N\\n\end{array}\right]\frac{(-1)^{n+1}(q)_nq^{n(n+1)/2}(1-q^{n^2})}{(q)_{n+N}(1-q^n)}.
\end{align}
The coefficients of the right-hand side always appear to be positive. It would be interesting to try to prove this. 
\section{Future directions}\label{cr}
Theorem \ref{sptgenthm} is a generalization of the generating function version of Andrews' famous identity for $\text{spt}(n)$ in that it involves an extra parameter $d$. It would be interesting to interpret the coefficients of $d^mq^n$ in the power-series expansions of each of the expressions in \eqref{sptgen}. We have not been able to do it as of now.

Note that one of the first steps in the proof of \eqref{sptgen} was to let $N\to\infty$ in \eqref{befmmain} to obtain \eqref{mmain}. So if one succeeds in extracting the arithmetic and combinatorial information embedded in \eqref{sptgen}, a similar thing could possibly be done beginning with the more general \eqref{befmmain}.

There are many further identities that could be derived from Theorem \ref{2phi1id} or its special cases considered in Section \ref{app}. For example, letting $d=q$ in Theorem \ref{sptgen} gives
\begin{align}
\sum_{j=1}^{\infty}\frac{q^{j^2}}{(q)_{j}^{2}}\sum_{n=1}^{j}\frac{q^n}{(1-q^{n+1})(1-q^n)}=\frac{q^2}{(1-q)^2(q)_{\infty}}.
\end{align}
Here, we have concentrated only on special cases of Theorem \ref{2phi1id} for $c$ or $d$ equal to $0,\pm1$, and that too only when $N\to\infty$!

Lastly, determining if all of the coefficients in the power series expansion of the right-hand side of \eqref{difference} are positive seems to be worthwhile to study in view of the fruitful consequences that ensue if it is true, for example, the inequality between $\overline{M}_1(n, N)$ and $\overline{M}_1(n, N)$ for all $n\in\mathbb{N}$ and hence the existence of the finite analogue of the $\textup{ospt}(n)$ function.



\end{document}